\documentclass[11pt,a4paper]{amsart}
\usepackage{amssymb,amsmath}

\numberwithin{equation}{section}
\newtheorem{theorem}{Theorem}[section]
\newtheorem{lemma}[theorem]{Lemma}
\newtheorem{proposition}[theorem]{Proposition}
\newtheorem{corollary}[theorem]{Corollary}

\theoremstyle{definition}
\newtheorem{definition}[theorem]{Definition}
\newtheorem{example}[theorem]{Example}
\newtheorem{problem}[theorem]{Problem}
\newtheorem{remark}[theorem]{Remark}

\newtheorem{acknowledgement}{Acknowledgement}

\newcommand\Supp{\operatorname{Supp}}
\newcommand\Ass{\operatorname{Ass}}
\newcommand\Assh{\operatorname{Assh}}
\newcommand\Ann{\operatorname{Ann}}

\newcommand\Hom{\operatorname{Hom}}

\newcommand\Ext{\operatorname{Ext}}
\newcommand\Rad{\operatorname{Rad}}

\newcommand\cd{\operatorname{cd}}
\newcommand\injdim{\operatorname{injdim}}

\newcommand\height{\operatorname{height}}

\newcommand{\gam}{\Gamma_{I}}
\newcommand{\Rgam}{{\rm R} \Gamma_{\mathfrak m}}

\newcommand{\qism}{\stackrel{\sim}{\longrightarrow}}

\begin{document}

\author[P.~Schenzel]{Peter Schenzel}
\title[Local cohomology modules]
{On Lyubeznik's invariants and endomorphisms of local cohomology modules}
\address{Martin-Luther-Universit\"at Halle-Wittenberg,
Institut f\"ur Informatik, D --- 06 099 Halle (Saale), Germany}
\email{peter.schenzel@informatik.uni-halle.de}

\begin{abstract}
Let $(R, \mathfrak m)$ denote an $n$-dimensional Gorenstein ring. For an
ideal $I \subset R$ of height $c$ we are interested in the endomorphism ring 
$B = \Hom_R(H^c_I(R), H^c_I(R)).$ It turns out that $B$ is a commutative 
ring. In the case of $(R,\mathfrak m)$ a regular local ring containing a 
field $B$ is a Cohen-Macaulay ring. Its properties are related to the 
highest Lyubeznik number $l = \dim_k \Ext_R^d(k,H^c_I(R)).$  In particular 
$R \simeq B$ if and only if $l = 1.$ Moreover, we show that
the natural homomorphism $\Ext_R^d(k, H^c_I(R)) \to k$ is non-zero.
\end{abstract}



\maketitle

\section{Introduction}
Let $(R, \mathfrak m, k)$ denote a local Noetherian ring. For an
ideal $I \subset R$ let $H^i_I(R), i \in \mathbb Z,$ denote the local
cohomology modules of $R$ with respect to $I$ (cf. \cite{aG} for the
definition). They carry several information about $I$ and $R.$
Their Bass numbers $\dim_k \Ext^j_R(k, H^i_I(R)), i, j \in \mathbb Z,$
are in various directions important invariants (cf. for instance
\cite{gL1}, \cite{gL2}, \cite{HS}). In the case of a regular local
ring they are investigated by Lyubeznik (cf. \cite{gL1}) known as Lyubeznik invariants.

On the other hand, in recent research there are sufficient conditions
when the endomorphism ring of $H^c_I(R), c = \height I,$ is isomorphic
to $R$ (cf. for instance \cite{HeS}). Note that it is not clear whether
it is a commutative ring in general. The endomorphism ring $B :=
\Hom_R(H^c_I(R), H^c_I(R))$ is the main subject of our investigations
here, in particular when $(R, \mathfrak m)$ is a Gorenstein ring. It
carries a lot of interesting properties. For an ideal $I \subset R$ let $I_d$ denote the intersection of the highest dimensional primary components of $I.$

\begin{theorem} \label{1.1} Let $(R, \mathfrak m)$ denote an $n$-dimensional
Gorenstein ring. Let $I \subset R$ be an ideal with $c = \height I$ and $d =
\dim R/I.$ Then:
\begin{itemize}
  \item[(a)] The endomorphism ring $B = \Hom_R(H^c_I(R), H^c_I(R))$ is commutative.
  There is a natural isomorphism \[\Hom_R(H^c_I(R), H^c_I(R)) \simeq \Ext_R^c(H^c_I(R), R).\]
   Moreover $B$ is $\mathfrak m$-adically complete provided $R$ is a complete local ring.
\end{itemize}
Suppose that $(R, \mathfrak m)$ is a complete regular local ring containing a field.
Assume that $d > 0.$ Then:
\begin{itemize}
  \item[(b)] The ring extension $R \subset B$ is module-finite and $B$ is a
   Noetherian ring.
  \item[(c)] $B$ is a free $R$-module of rank $\dim_k \Ext_R^c(k, H^c_I(R))$ and therefore
  $B$ is a Cohen-Macaulay module.
  \item[(d)] $B$ is a local Noetherian ring if and only if $V(I_d)$ is connected
  in codimension one.
  \item[(e)] $R \to \Hom_R(H^c_I(R), H^c_I(R)) = B$ is an isomorphism if and only if
  the completion of the strict Henselization of $R/I_d$ is connected in codimension one.
\end{itemize}
\end{theorem}

By the aid of the result shown by Huneke and Lyubeznik (cf. \cite[Theorem 2.9]{HL}) about the
second Vanishing Theorem of Hartshorne (cf. \cite[Theorem 7.5]{rH3}) we get as a consequence of Theorem \ref{1.1}:

\begin{corollary} \label{1.2} Let $(R, \mathfrak m)$ is a complete $n$-dimensional regular
local ring containing a field. Let $I \subset R$ denote an ideal and $d = \dim R/I \geq 2.$
Then the following conditions are equivalent:
\begin{itemize}
  \item[(i)] The natural homomorphism $R \to \Hom_R(H^c_I(R), H^c_I(R))$ is an isomorphism.
  \item[(ii)] $H^i_I(R) = 0$ for all $i \geq n - 1.$
\end{itemize}
\end{corollary}

In the general case of a Gorenstein ring it is not true (cf. \ref{3.5}) true that
$B$ is a module finite extension, while -- in this example -- it is a Noetherian ring. We conjecture that $B$ is always a Noetherian ring (cf. \ref{3.5}).

Another feature of our considerations here is the the following natural map
\[
\phi : \Ext_R^d(k, H^c_I(R)) \to k,
\]
where $I \subset R$ is an ideal of $\height I = c$ and $d = \dim R/I.$ The homomorphisms
$\phi$ occurs as an edge homomorphism of a certain spectral sequence resp. as a homomorphism
in the construction of the truncation complex (cf. Definition \ref{5.2} and Lemma \ref{6.1}).
In \cite[Conjecture 2.7]{HeS} Hellus and the author conjectured that it is always non-zero. Here we show that $\phi$ is related to the homomorphism $\lambda : \Ext_R^d(k, H^c_I(R)) \to H^d_{\mathfrak m}(H^c_I(R)). $ As an application of our techniques there is the following result:

\begin{theorem} \label{1.3} Let $(R, \mathfrak m)$ be an $n$-dimensional Gorenstein ring.
Let $I \subset R$ denote an ideal with $c = \height I$ and $d= \dim R/I.$ Then:
\begin{itemize}
  \item[(a)] If $\phi : \Ext_R^d(k, H^c_I(R)) \to k$ is non-zero, then $\lambda : \Ext_R^d(k, H^c_I(R))
   \to  H^d_{\mathfrak m}(H^c_I(R))$ is non-zero.
\end{itemize}
Suppose that $(R, \mathfrak m)$ is a complete regular local ring containing a field. Then:
\begin{itemize}
  \item[(b)] The homomorphism  $\phi : \Ext_R^d(k, H^c_I(R)) \to k$ is non-zero.
\end{itemize}
\end{theorem}

In a certain sense (cf. Remark \ref{6.3}) one might consider the homomorphism $\lambda$ as the
limit version of the homomorphism $\Ext_R^d(k, K(R/I^{\alpha})) \to H^d_{\mathfrak m}(K(R/I^{\alpha}))$ as it was studied by Hochster (cf. \cite[Section 4]{mH}). Here $K(R/I^{\alpha})$ denotes the canonical module of $R/I^{\alpha}, \alpha \in \mathbb N.$ For the proof of Theorem \ref{1.3} we refer to Section 6 and Theorem \ref{6.2}.

As another application of our technique we proof a slight sharpening of Blickle's result (cf.
\cite[Theorem 1.1]{mB}) about a certain duality for the Lyubeznik numbers. In the second Section of the paper there are preliminaries and auxiliary results needed in the sequel. In the third section we investigate the endomorphism ring of $H^c_I(R), c = \height I,$ in the case of $R$ a Gorenstein ring and in Section four in the case of $R$ a regular ring containing a field. To this end me make use of the deep results of Huneke and Sharp (cf. \cite{HS}) in the case of prime characteristic $p > 0$ and of Lyubeznik (cf. \cite{gL1}) in characteristic zero. In Section 5 there are some additional results about the so-called Lyubeznik numbers. In Section 6 we study the natural homomorphism $\phi : \Ext_R^d(k, H^c_I(R)) \to k.$ In the final section 7 we discuss some additional examples.

\section{Preliminaries and auxiliary results}
In this section we will summarize a few results
about completions of local rings and inverse limits as well as some
statements about the canonical module of a module. They are
needed for further constructions related to local cohomology
modules and their asymptotic behavior. With our notation
we follow the textbook \cite{hM}.

Let $\{I_i\}_{i \in \mathbb N}$ denote a family of ideals of a
Noetherian ring $R$ such that $I_{i+1} \subset I_i$ for all $i
\geq 1.$ With the natural epimorphism $R/I_{i+1} \to R/I_i$ the
family $\{R/I_i\}_{i \in \mathbb N}$ forms an inverse system. Its
inverse limit $\varprojlim R/I_i$ is given by
\[
\varprojlim R/I_i \simeq \{ (b_i+I_i)_{i \in \mathbb N} :
b_i \in R, b_{i+1} - b_i \in I_i \text{ for all } i \in \mathbb N\}.
\]
In the following result there is a summary of some technical constructions.

\begin{lemma} \label{2.1} Let $(R,\mathfrak m)$ denote a
complete local ring. Let $\{I_i\}_{i \in \mathbb N}$ denote
a descending sequence of ideals of $R.$
\begin{itemize}
\item[(a)] Assume that $\cap_{i \in \mathbb N} I_i = 0.$ Then
for any $i \in \mathbb N$ there exists an integer $j = j(i) \geq
i$ such that $j(i+1) > j(i)$ and $I_{j(i)} \subset \mathfrak m^i.$
\item[(b)] Suppose that for all $i \in \mathbb N$ there is an integer $j =
j(i)$ with $j \geq i, j(i+1) > j(i),$ and $I_j \subset \mathfrak
m^i.$ Then $\varprojlim R/I_i \simeq R.$
\end{itemize}
\end{lemma}

\begin{proof} The statement in (a) is a consequence of Chevalley's
Lemma (cf. e.g. \cite[Theorem 2.1]{pS1}).

For the proof of (b) let $(b_i+I_i)_{i \in \mathbb N} \in \varprojlim
R/I_i$ denote a given element. Therefore $b_{i+1} - b_i \in I_i$
for all $i \in \mathbb N.$ By view of the assumption it turns out that for every
$i \in \mathbb N$ there is an $j = j(i)$ such that $I_{j(i)} \subset
\mathfrak m^i.$ Therefore we see that $b_{j(i+1)} - b_{j(i)} \in
I_{j(i)} \subset \mathfrak m^i$ for all $i \in \mathbb N.$ that is,
\[
(b_{j(i)} + \mathfrak m^i)_{i \in \mathbb N} \in \varprojlim
R/\mathfrak m^i \simeq R
\]
because $R$ is complete with respect to the $\mathfrak m$-adic
topology. That is, $(b_{j(i)} + \mathfrak m^i)_{i \in \mathbb
N}$ defines an element $b \in R.$ Consequently  $b_{j(i)} -
b \in \mathfrak m^i$ for all $i \in \mathbb N.$ Now fix $l \in
\mathbb N$ and choose $i \geq l.$ Then $b_{j(i)} - b_{j(l)} \in
I_{j(l)}$ as it follows by induction since $j(i) \geq j(l).$ So
there is the following relation $b - b_{j(l)} \in
(\mathfrak m^i, I_{j(l)})$ for all $i \geq l.$ By the Krull
Intersection Theorem it provides that $b - b_{j(l)} \in I_{j(l)}$ for
any $l \geq 1.$ But this means nothing else but
\[
b = (b + I_{j(l)})_{l \in \mathbb N} = (b_{j(l)} + I_{j(l)})_{l \in
\mathbb N} \in \varprojlim R/I_{j(l)} \simeq \varprojlim R/I_i,
\]
so that $( b_i + I_i)_{i \in \mathbb N} \in R$ and $\varprojlim R/I_i \simeq R,$
as required.
\end{proof}

For an application we have to know whether the inverse limit of an
inverse system of rings is a quasi-local ring. Here a commutative ring is
called quasi-local provided there is a unique maximal ideal.

\begin{lemma} \label{2.2} Let $\{(B_i,\mathfrak n_i)\}_{i \in
\mathbb N}$ denote an inverse system of local rings such that
\[
\phi_{i+1} : (B_{i+1}, \mathfrak n_{i+1}) \to (B_i, \mathfrak
n_i), i \in \mathbb N,
\]
is a local homomorphism. Then
\[
B : = \varprojlim (B_i, \mathfrak n_i) \simeq \{( b_i)_{i \in
\mathbb N} : b_i \in B_i, \phi_{i+1}(b_{i+1}) = b_i \text{ for all
} i \in \mathbb N \}
\]
is a quasi-local ring.
\end{lemma}

\begin{proof} It is easily seen that $B$ admits the structure of a
commutative ring with identity element $(1)_{i \in \mathbb N} \in B.$ Let
$(b_i)_{i \in \mathbb N} \in B$ denote a unit. By definition there
exists an element $(a_i)_{i \in \mathbb N} \in B$ such that
\[
(a_ib_i)_{i \in \mathbb N} = (1)_{i \in \mathbb N}.
\]
This means $a_ib_i = 1$ for all $i \in \mathbb N,$ so that $b_i$
is a unit in $(B_i, \mathfrak n_i)$ for all $i \in \mathbb N.$ On
the other hand let $(b_i)_{i \in \mathbb N} \in B$ denote an
element such that for all $i \in \mathbb N$ the element $b_i \in
B_i$ is a unit. Then there is a sequence $a_i \in B_i, i \in
\mathbb N,$ of elements such that $a_i b_i = 1.$ We claim that
$(a_i)_{i \in \mathbb N} \in B.$ To this end we have to show that
$\phi_{i+1}(a_{i+1}) = a_i$ for all $i \in \mathbb N.$ Because
$\phi_{i+1}$ is a homomorphism of rings
\[
1 = \phi_{i+1}(a_{i+1}) \phi_{i+1}(b_{i+1}) = \phi_{i+1}(a_{i+1})
b_i {\text{ and }} 1 = a_i b_i,
\]
so that $\phi_{i+1}(a_{i+1}) = a_i$ because $b_i$ is a unit in $B_i.$

Now let $(b_i)_{i \in \mathbb N} \in B$ be an element such that
$b_j \in B_j$ is not a unit for some $j \in \mathbb N.$ Then $b_j
\in \mathfrak n_j$ and $b_i \in \mathfrak n_i$ for all $i \in
\mathbb N$ as easily seen since $\phi_i$ is a local homomorphism
of local rings for all $i \in \mathbb N.$ Therefore
\[
B \setminus B^{\star} = \{(b_i)_{i \in \mathbb N} \in B : b_i \in
\mathfrak n_i {\text{ for all }} i \in \mathbb N \},
\]
where $B^{\star}$ denotes the set of units of $B.$ Moreover,
$\{(b_i)_{i \in \mathbb N} \in B : b_i \in \mathfrak n_i \text{
for all } i \in \mathbb N \}$ is an ideal of $B.$ Because
$B\setminus B^{\star}$ is equal to the union of maximal ideals of
$B$ it follows that $B \setminus B^{\star}$ itself is a maximal
ideal. But this means nothing else but $B$ is a quasi-local ring.
\end{proof}

Now let $(R,\mathfrak m)$ denote a local
Gorenstein ring with $n = \dim R.$ Let $M$ be a finitely generated
$R$-module and $d = \dim M.$ Define
\[
K^i(M) := \Ext_R^{n-i}(M,R), i \in \mathbb Z,
\]
the $i$-th module of deficiency. For $i = d$ let
\[
K(M) = K^d(M) = \Ext_R^c(M,R), c = n-d,
\]
denote the canonical module of $M.$ Let $H^i_{\mathfrak m}(\cdot)$ denote the
$i$-th local cohomology functor with support in $\mathfrak m.$ By
the Local Duality Theorem (see e.g. \cite{aG} or \cite[Theorem
1.8]{pS1}) there are the following functorial isomorphisms
\[
H^i_{\mathfrak m}(M) \simeq \Hom_R(K^i(M), E), i \in \mathbb Z,
\]
where $E = E_R(R/\mathfrak m)$ denotes the injective hull of
$k = R/\mathfrak m,$ the residue field.

For some basic properties  about the modules of deficiency we refer to
\cite{aG} and \cite[Lemma 1.9]{pS2}.
In particular, $\Ann_R K(M) = (\Ann _R M)_d, $ the intersection of all
the $\mathfrak p$-primary components of $\Ann_R M$ such that $\dim
R/\mathfrak p = d.$ Moreover $K(M)$ satisfies Serre's condition $S_2.$

For an $R$-module $M$ we will consider $K(K(M)).$ To this end we need
the following construction.

\begin{proposition} \label{2.3} There is a  canonical isomorphism
\[
K(K(M)\otimes M) \simeq \Hom(K(M), K(M)) \simeq \Hom(M, K(K(M)))
\]
for a finitely generated $R$-modules $M.$
\end{proposition}

\begin{proof} Let $R \to E^{\cdot}$ denote a minimal injective
resolution of $R.$ Then it induces an exact sequence
\[
0 \to K(M) \to \Hom(M,E^{\cdot})^c \to \Hom(M, E^{\cdot})^{c+1},
\]
where $c = \dim R - \dim M.$ To this end recall that $\Hom_R(M,E^{\cdot})^i = 0$
for all $i < c.$ By applying the functor $\Hom(K(M),
\cdot)$ it yields an exact sequence
\[
0 \to \Hom(K(M),K(M)) \to \Hom(K(M), \Hom(M,E^{\cdot}))^c \to
\Hom(K(M), \Hom(M, E^{\cdot}))^{c+1}.
\]
By the adjunction formula and the definition it provides that
\[
K(K(M) \otimes M) \simeq \Hom(K(M), K(M)).
\]
Because of $\dim K(M) = \dim M$ a similar argument provides the proof
of the second isomorphism.
\end{proof}

Of a particular interest of \ref{2.3} is the case of $R/I = A$
for an ideal $I \subset R.$

\begin{proposition} \label{2.4} Let $I$ denote an ideal of the
Gorenstein ring R. For $A = R/I$ there are the following results:

\begin{itemize}
\item[(a)] $K(K(A)) \simeq \Ext_R^c(\Ext_R^c(A,R),R), c = \dim R -
\dim A.$

\item[(b)] $K(K(A))$ is isomorphic to the endomorphism ring
$\Hom(K(A), K(A))$ which is commutative.

\item[(c)] There is an injection $A/0_d \hookrightarrow \Hom(K(A),
K(A)).$

\item[(d)] $\Hom(K(A), K(A))$ is the $S_2$-fication of $A/0_d.$
\end{itemize}
\end{proposition}

For the proof we refer to \cite{HH}, \cite{pS2} and the previous
Proposition \ref{2.3}. Next we recall a definition introduced by
Hochster and Huneke (cf. \cite[(3.4)]{HH}).

\begin{definition} \label{2.5} Let $R$ denote an
commutative Noetherian ring with finite dimension. We denote by
$\mathbb G_R$ the undirected graph whose vertices are primes
$\mathfrak p$ of $R$ such that $\dim R = \dim R/\mathfrak p,$ and
two distinct vertices $\mathfrak p, \mathfrak q$ are joined by
an edge if and only if $(\mathfrak p, \mathfrak q)$ is an ideal of
height one.
\end{definition}

By view of the definition \ref{2.5} observe the following: For a
commutative ring $R$ with $\dim R < \infty$ the variety $V(0_d)$
is connected in codimension one if and only if $\mathbb G_R$ is
connected. For the notion of connectedness in codimension one we
refer to Hartshorne's paper \cite[Proposition 1.1 and
Definition]{rH}.

As an application here we describe when the
endomorphism ring $K(K(A)) \simeq \Hom(K(A), K(A))$ of the canonical module
$K(A)$ is a local ring, see \cite[3.6]{HH}.

\begin{lemma} \label{2.6} With the notion of Proposition
\ref{2.4} and assuming that $R$ is complete the following
conditions are equivalent:

\begin{itemize}

\item[(i)] $K(A)$ is indecomposable.

\item[(ii)] $V(0_d)$ is connected in codimension one.

\item[(iii)] $K(K(A))$ is a local ring.
\end{itemize}
\end{lemma}

\section{On a formal ring extension}
Let $(R, \mathfrak m)$ denote an $n$-dimensional Gorenstein ring. Then it is well-known
(cf \cite{rH}) that the minimal injective resolution $R \qism E^{\cdot}$ is
a dualizing complex. In particular, this provides a natural homomorphism
\[
M \to \Ext_R^c(\Ext_R^c(M, R), R), \quad c = n - \dim M,
\]
for a finitely generated $R$-module M. Let $I \subset R$ denote
an ideal of $R$ and put $A = R/I.$ So there is a natural homomorphism
\[
A \to \Ext^c_R(\Ext^c_R(A,R), R) \simeq \Hom_R(K(A), K(A)), \quad c = n - d.
\]
It is well-known (cf. e.g. \cite{pS2}) that the kernel coincides with
the ideal $I_d,$ that is the intersection of all primary components
of $I$ whose dimension is $d = \dim R/I.$ In other words, $I_d = I R_S \cap R,$ where
$S = \cap_{\mathfrak p \in \Assh R/I} R \setminus \mathfrak p$ and $\Assh M =
\{\mathfrak p \in \Ass M | \dim R/\mathfrak p = \dim M\}$ for an $R$-module $M.$
Moreover, $\Hom_R(K(A), K(A))$ is a commutative ring, the $S_2$-fication of $A$ and
the natural homomorphism above is a homomorphism of commutative rings.

Now let $R/I^{\alpha +1} \to R/I^{\alpha}, \alpha \in \mathbb N,$ denote the
natural homomorphism. Then there is a commutative diagram
\[
\begin{array}{ccc}
  R/I^{\alpha +1} & \to & \Ext_R^c(\Ext_R^c(R/I^{\alpha +1},R),R) \\
  \downarrow &  & \downarrow \\
  R/I^{\alpha} & \to & \Ext_R^c(\Ext_R^c(R/I^{\alpha},R),R).
\end{array}
\]
Clearly, for each $\alpha \in \mathbb N$ the vertical homomorphisms are homomorphisms
of commutative rings.

\begin{theorem} \label{3.1} With the previous notation there is a homomorphism
\[
\phi : {\hat R}^I \to  \varprojlim \Ext_R^c(\Ext_R^c(R/I^{\alpha},R),R) =: B,
\]
where ${\hat R}^I$ denotes the $I$-adic completion of $R.$ Moreover, there are
the following properties:
\begin{itemize}
\item[(a)] $B$ admits the structure of a commutative ring so that $\phi$
is a homomorphism of rings.
\item[(b)] $\ker \phi = \cap_{\alpha \geq \mathbb N} (I^{\alpha})_d = 0_S,$
where $S = \cap_{\mathfrak p \in \Assh R/I} R \setminus \mathfrak p.$
\item[(c)] The composition of the natural homomorphism $R \to {\hat R}^I$ with $\phi: {\hat R}^I \to B$
induces a non-zero homomorphism $k \otimes R \to k \otimes B.$ In particular $B \not= 0.$
\end{itemize}
\end{theorem}

\begin{proof} The inverse system above defined by $R/I^{\alpha}  \to \Ext_R^c(\Ext_R^c(R/I^{\alpha},R),R)$ provides -- by passing to the inverse
limit -- the homomorphism $\phi.$ As discussed above
\[
 \Ext_R^c(\Ext_R^c(R/I^{\alpha},R),R) \simeq \Hom_R(K(R/I^{\alpha}), K(R/I^{\alpha}))
 \]
admits the structure of a commutative ring such that the corresponding homomorphisms
in the inverse system are homomorphism of rings (cf. \cite{HH}). Then it is easy to check that
\[
\varprojlim \Hom_R(K(R/I^{\alpha}), K(R/I^{\alpha}))
\]
admits the structure of a commutative ring such that $\phi$ is a homomorphism of rings.
This proves (a).

For the proof of (b) recall that the kernel of $R/I^{\alpha} \to \Hom_R(K(R/I^{\alpha}), K(R/I^{\alpha}))$ is equal to $(I^{\alpha})_d.$ Then $\ker \phi = \cap_{\alpha \geq \mathbb N} (I^{\alpha})_d,$ as follows by elementary properties of the inverse limit.
Moreover $\ker \phi = 0_S$ by the Krull Intersection Theorem and the discussion above.

For the proof of (c) notice that the homomorphisms $R \to {\hat R}^I$ as well as $\phi : {\hat R}^I \to B$ are ring homomorphisms. That is, they respect the identity. Therefore, the
residue class $1 + \mathfrak m$ does not map to zero.
\end{proof}

Now let us relate the structure of the ring $B = \varprojlim
\Ext_R^c(\Ext_R^c(R/I^{\alpha},R),R)$ to the local cohomology of $R$ with
respect to $I.$ Surprisingly the Matlis dual of $H^d_{\mathfrak
m}(H^c_I(R))$ admits the structure of a commutative ring.

\begin{theorem} \label{3.2} Let $(R, \mathfrak m)$ denote a
Gorenstein ring. Let $I \subset R$ be an ideal with $d = \dim R/I$ and $c = \height I.$
\begin{itemize}
\item[(a)] There are natural isomorphisms
\[
B \simeq \Ext_R^c(H^c_I(R), R) \simeq \Hom_R(H^c_I(R), H^c_I(R)).
\]
\item[(b)] If $R$ is in addition complete, then
\[
\varprojlim \Ext_R^c(\Ext_R^c(R/I^{\alpha},R),R) \simeq \Hom_R(H^d_{\mathfrak m}(H^c_I(R)), E).
\]
\item[(c)] If $R$ is complete, then $B$ is also $\mathfrak m$-adically complete.
\item[(d)] If $V(I_d)$ is connected in codimension one, then $(B, \mathfrak n)$ is a quasi-local ring.
\end{itemize}
\end{theorem}

\begin{proof}  By definition $H^c_I(R) \simeq \varinjlim \Ext_R^c(R/I^{\alpha}, R),$ so that
$B \simeq \Ext_R^c(H^c_I(R), R).$ To this end recall that $\Ext_R^{\cdot}(\cdot,R)$ transforms a direct system into an inverse system.
Let $R \qism E^{\cdot}$ denote a minimal injective resolution. Then there is an exact sequence $0 \to H^c_I(R) \to \gam(E^{\cdot})^c \to \gam(E^{\cdot})^{c+1}.$ It induces a natural commutative diagram with exact rows
\[
  \begin{array}{cccccc}
    0 \to & \Hom_R(H^c_I(R), H^c_I(R)) & \to & \Hom_R(H^c_I(R),\gam(E^{\cdot}))^c  & \to  & \Hom_R(H^c_I(R),\gam(E^{\cdot}))^{c+1} \\
      &     \downarrow             &      & \downarrow                          &      & \downarrow \\
    0 \to & \Ext_R^c(H^c_I(R), R)        & \to  & \Hom_R(H^c_I(R), E^{\cdot})^c       & \to  & \Hom_R(H^c_I(R), E^{\cdot})^{c+1} \\
  \end{array}
\]
because $\gam(E^{\cdot})$ is a subcomplex of $E^{\cdot}.$ The two last vertical homomorphisms
are isomorphisms. This follows because $\Hom_R(X, E_R(R/\mathfrak p)) = 0$ for an $R$-module $X$ with $\Supp_R X \subset V(I)$ and $\mathfrak p \not\in V(I).$ Therefore the first vertical map is also an isomorphism.

For the proof of (b) recall that the local cohomology commutes with direct limits. So, by the definition of $H^c_I(R)$ there is an isomorphism
\[
\Hom(H^d_{\mathfrak m}(H^c_I(R)), E) \simeq \varprojlim
\Hom(H^d_{\mathfrak m}(\Ext_R^c(R/I^{\alpha}, R)),E)
\]
and the Local Duality Theorem provides the statement in (b).

By virtue of (a) there is the natural isomorphism
\[
\varprojlim B/\mathfrak m^{\alpha}B \simeq \varprojlim (R/\mathfrak m^{\alpha}
\otimes \Hom(H^d_{\mathfrak m}(H^c_I(R)), E)).
\]
Since $E$ is an injective $R$-module there are the following
natural isomorphisms
\[
R/\mathfrak m^{\alpha} \otimes \Hom(H^d_{\mathfrak m}(H^c_I(R)), E)
\simeq \Hom(\Hom(R/\mathfrak m^{\alpha}, H^d_{\mathfrak m}(H^c_I(R))), E)
\]
for all $\alpha \in \mathbb N.$ As a consequence there is the
isomorphism
\[
\varprojlim B/\mathfrak m^{\alpha}B \simeq \Hom(H^0_{\mathfrak
m}(H^d_{\mathfrak m}(H^c_I(R))), E).
\]
But $H^d_{\mathfrak m}(H^c_I(R))$ is a module whose support is
contained in $V(\mathfrak m)$ so that
\[
H^0_{\mathfrak
m}(H^d_{\mathfrak m}(H^c_I(R))) \simeq H^d_{\mathfrak
m}(H^c_I(R)).
\]
But this means nothing else but $\varprojlim
B/\mathfrak m^{\alpha}B \simeq B.$ So, (c) is true.

For the proof of (d) let $\alpha \in \mathbb N$ be an integer. Define
\[
B_{\alpha} = \Ext_R^c(\Ext_R^c(R/I^{\alpha},R),R) =
\Hom_R(K(R/I^{\alpha}), K(R/I^{\alpha}))
\]
the endomorphism ring of the canonical module of
$R/I^{\alpha}.$ Since $
V(I) = V(I^{\alpha}), \alpha \in \mathbb N,$
is connected in codimension one we know (cf. Lemma \ref{2.6}) that
$(B_{\alpha}, \mathfrak n_{\alpha}), \alpha \in \mathbb N,$ is a local
ring. By view of Lemma \ref{2.2} it
follows that $(B, \mathfrak n) \simeq \varprojlim (B_{\alpha}, {\mathfrak n}_{\alpha})$ 
is a quasi-local ring. To this end note $(B_{\alpha +1}, \mathfrak n_{\alpha +1}) \to 
(B_{\alpha}, n_{\alpha})$ is a local homomorphism. This follows by contracting the 
maximal ideal $\mathfrak n_{\alpha}$ along the commutative diagram before 
Theorem \ref{3.1}. On the left it provides an injection $ k = k \hookrightarrow 
B_{\alpha}/{\mathfrak n_{\alpha}}.$ Therefore, on the right it yields 
$k \hookrightarrow B_{\alpha +1}/{\mathfrak n_{\alpha}} \cap B_{\alpha +1}.$ Because 
$B_{\alpha +1}$ is finitely generated over $R/I^{\alpha + 1}$ it follows  
that ${\mathfrak n}_{\alpha} \cap B_{\alpha + 1} = {\mathfrak n_{\alpha + 1}}.$ 
\end{proof}

\begin{problem} \label{3.3}
It is a natural question to ask whether the commutative ring $B$
constructed in Theorem \ref{3.1} is a Noetherian ring. We do not
know an answer in general. The stronger question whether $B$ is a
finitely generated $R$-module is not true.

To this end observe the following: Suppose that $R$ is complete.
Because of Theorem \ref{3.1} $B$ is a finitely generated $R$-module
if and only if $\dim_k B/\mathfrak mB < \infty$ (cf. e.g. \cite[Theorem
8.4]{hM}).
\end{problem}

\begin{lemma} \label{3.4} Fix the notation of Theorem \ref{3.2}. Assume in
addition that $R$ is complete. Then the following conditions are equivalent:
\begin{itemize}
\item[(i)] $\Ext^c_{\hat R}(H^c_{I \hat R}(\hat R),\hat R)$ is a finitely generated $\hat R$-module.

\item[(ii)] $\dim_k \Hom(k, H^d_{\mathfrak m}(H^c_I(R))) <
\infty.$

\item[(iii)] $H^d_{\mathfrak m}(H^c_I(R))$ is an Artinian
$R$-module.
\end{itemize}
\end{lemma}

\begin{proof} The proof is clear by view of Matlis duality and the
previous observation.
\end{proof}

In the following there is an example of an ideal $I \subset R$ in
a local complete Gorenstein ring $R$ such that $B$ is not a
finitely generated $R$-module.

\begin{example} \label{3.5} Let $k$ be a field and let
$A = k[|u,v,x,y|]$ be the formal power series ring in four
variables. Put $R = A/fA,$ where $f = xv-yu.$ Let $I = (x,y)R.$ We
will show that
\[
B = \varprojlim \Ext^1_R(\Ext^1_R(R/I^{\alpha}, R),R)
\]
is not a finitely generated $R$-module, while it is a Noetherian ring.

To this end put $A_{\alpha} = R/I^{\alpha} \simeq A/((x,y)^{\alpha}, f)$ and 
$B_{\alpha} = k[|u,v|][a]/(a^{\alpha}),$ where $a$ denotes a variable over $k[|u,v|].$ 
Consider the ring homomorphism $A \to B_{\alpha}$ induced by $x \mapsto ua, y \mapsto
va.$ As it is easily seen it induces an injection $A_{\alpha} \to B_{\alpha}, \alpha
\in \mathbb N.$ Clearly $B_{\alpha}, {\alpha} \in \mathbb N,$ is a
two-dimensional Cohen-Macaulay ring. The cokernel of this embedding is 
\[ 
k[|u,v|][a]/(k[|u,v|][ua,va] + a^{\alpha}k[|u,v|][a]),
\]
which is a finite dimensional $k$-vector space. Whence $\dim B_{\alpha}/A_{\alpha} = 0$ 
and $B_{\alpha}$ is the $S_2$-fication of $A_{\alpha},$ that is  
$B_{\alpha} \simeq \Ext^1_R(\Ext^1_R(R/I^{\alpha},R),R)$ 
(cf. Proposition \ref{2.4}).

Therefore there are short exact sequences
\[
0 \to A_{\alpha} \to B_{\alpha} \to H^1_{\mathfrak m}(A_{\alpha}) \to 0
\]
for all $\alpha \in \mathbb N.$ By passing to the inverse limit it
induces a short exact sequence
\[
0 \to R \to B \to \varprojlim H^1_{\mathfrak m}(R/I^{\alpha}) \to 0.
\]
Moreover it provides that $B \simeq k[|u,v,a|],$ which is clearly
a Noetherian ring. Moreover $B$ is not a finitely generated
$R$-module as easily seen.

Furthermore, by the local duality theorem there is the isomorphism
\[
\varprojlim H^1_{\mathfrak m}(R/I^{\alpha}) \simeq \Hom_R(H^2_I(R), E).
\]
Therefore, $\Hom_R(H^2_I(R),E)$ is not a finitely generated
$R$-module. By Matlis duality it follows that $H^2_I(R)$ is not an
Artinian $R$-module and therefore the socle $\Hom_R(k, H^2_I(R))$
is not of finite dimension. Originally this was shown by
Hartshorne (cf. \cite[Section 3]{rH2}). In fact, the analysis of
Hartshorne's example inspired the above construction.  
\end{example}

The proof of Theorem \ref{1.1} (a) follows by Theorem \ref{3.1} (a) and Theorem \ref{3.2} (b).  In the following sections we shall discuss some particular cases
in which $B$ is a finitely generated $R$-module.

\section{The case of regular local rings}
In this section let $(R,\mathfrak m)$ denote a regular local ring.
Let $I \subset R$ be an ideal of $R.$ Huneke and Sharp (cf.
\cite{HS}) in the case of prime characteristic $p > 0$ resp.
Lyubeznik (cf. \cite{gL1}) in the case of characteristic zero
proved the following result:

\begin{remark} \label{4.1} Let $(R,\mathfrak m)$ be a regular
local ring containing a field. Let $I \subset R$ be an ideal. For
all $i,j \in \mathbb Z$ the following (cf. \cite{HS} and \cite{gL1})
was shown:
\begin{itemize}
\item[(a)] $H^j_{\mathfrak m}(H^i_I(R))$ is an injective
$R$-module,

\item[(b)] $\injdim_R H^i_I(R) \leq \dim_R H^i_I(R) \leq \dim R
-i.$

\item[(c)] $\Ext^j_R(k, H^i_I(R)) \simeq \Hom_R(k, H^j_{\mathfrak m}(H^i_I(R))$ and $\dim_k \Ext_R^j(k, H^i_I(R)) < \infty.$

\end{itemize}
\end{remark}

As we shall see the previous result applies in an essential way in
order to describe properties of the ring
\[
B = \varprojlim \Ext_R^c(\Ext_R^c(R/I^i,R),R),\; c = \dim R - \dim
R/I,
\]
introduced in Section 4. But before let us recall the Bass numbers
in Remark \ref{4.1} (c) were introduced by Lyubeznik (cf.
\cite[4.1]{gL1}). In fact Lyubeznik has shown that they only depend
upon $R/I.$ With the above results in mind we shall describe the structure of
the ring $B$ in case $(R,\mathfrak m)$ is a complete regular local
ring containing a field.

\begin{lemma} \label{4.2} With the notation as above assume that
$(R,\mathfrak m)$ is a complete regular local ring containing a
field. There is an isomorphism
\[
\varprojlim \Ext_R^c(\Ext_R^c(R/I^i,R),R) \simeq R^l,
\]
where $l = \dim \Ext_R^d(k, H^c_I(R)), d = \dim R/I, c = \dim R
-\dim R/I.$
\end{lemma}

\begin{proof} By virtue of Remark \ref{4.1} it turns out that $l$
is a finite number. As a consequence of
Lemma \ref{3.4} it follows that $B$ is a finitely generated
$R$-module. Moreover, by virtue of Remark \ref{4.1} (a) and the
definition of $B$ as the Matlis dual of $H^d_{\mathfrak
m}(H^c_I(R))$ (cf. Theorem \ref{3.2}) we see that $B$ is flat as an
$R$-module. Therefore $B$ is a free $R$-module and $B \simeq R^l,$
the direct sum of $l$ copies of $R.$
\end{proof}

In the following we will give an interpretation of the rank $l$ in
topological terms. To this end we use results of Lyubeznik (cf. \cite{gL2})
and Zhang (cf. \cite{wZ}).

\begin{remark} \label{4.3} Let $(R, \mathfrak m)$ be an $n$-dimensional regular local
ring containing a field. Let $I \subset R$ denote an ideal with $c = \height I$
and $d = \dim R/I.$ Let $A$  denote the completion of the strict
Henselization of the completion of $R/I.$ Let $t$ denote the
number of connected components of the graph $\mathbb G_A.$ Then
\[
\dim_k \Hom_R(k, H^d_{\mathfrak m}(H^c_I(R))) = \dim_k \Ext_R^d(k,
H^c_I(R)) = t
\]
(cf. \cite{gL2} and \cite{wZ}).
\end{remark}

It was pointed out by Lyubeznik (cf. \cite{gL2}) that the graph
$\mathbb G_{A},$ where $A$ the completion of the strict
Henselization of the completion of $R/I,$ is realized
by a smaller ring. Namely, let $k \subset \hat{A}$ denote a
coefficient field. It follows (cf. \cite[Theorem 4.2]{HL}) that there
exists a finite separable field extension $k \subset K$ such that
the graphs $\mathbb G_{K \otimes_k \widehat{R/I}}$ and $\mathbb G_A$ are
isomorphic.

Here we want to give another description of the invariant $t = \dim_k \Ext_R^d(k,
H^c_I(R)).$ As a first step in this direction there is the
following result:

\begin{theorem} \label{4.4} Let $(R, \mathfrak m)$ denote an $n$-dimensional
complete regular ring containing a field. Let $I$ be an ideal of
$R$ and $c = \height I.$ Then the following is equivalent:
\begin{itemize}
\item[(i)] $V(I_d)$ is connected in codimension one.
\item[(ii)]
$ B = \varprojlim \Ext_R^c(\Ext_R^c(R/I^{\alpha},R),R)$ is a local ring.
\end{itemize}
\end{theorem}

\begin{proof} Because of $B \simeq \Hom_R(H^d_{\mathfrak m}(H^c_I(R), E)$ we may 
replace $I$ by $I_d.$ The implication (i) $\Longrightarrow$ (ii) is a consequence of
the results shown in Theorem \ref{3.2} (d) and Lemma \ref{4.2}, where it follows that
$B$ is a Noetherian ring.

In order to prove (ii) $\Longrightarrow$ (i) suppose that $V(I)$ is not connected
in codimension one. Let $\mathbb G_1, \ldots, \mathbb G_t,$$ t > 1,$ denote the connected
components of $\mathbb G_{R/I}.$ Moreover, let $I_i, i = 1,\ldots,t,$ denote the intersection of all
minimal primes of $V(I)$ that are vertices of $\mathbb G_i.$ Then
\[
H^c_I(R) \simeq \oplus_{i = 1}^t H^c_{I_i}(R)
\]
as it is a consequence of the Mayer-Vietoris
sequence for local cohomology (cf.
\cite[Proposition 2.1]{gL2} for the details). Clearly $c = \height I_i, i = 1, \ldots, t,$ and
$H^c_{I_i}(R) \not= 0.$ Moreover
\[
\Ext_R^c(H^c_I(R), R) \simeq \oplus_{i=1}^t \Ext_R^c(H^c_{I_i}(R), R)
\]
and therefore $B \simeq B_1 \times \ldots \times B_t,$ where $B_i = \Ext_R^c(H^c_{I_i}(R), R),  i= 1,\ldots, t,$ are local rings as shown by (i) $\Longrightarrow$ (ii). Because of $t > 1$ it follows that $B$ is not a local ring.
\end{proof}

As a corollary of the previous statement we are able to describe the number of connected
components of $\mathbb G_{R/I}$ in terms of  the ring structure of $B = \Ext_R^c(H^c_I(R), R).$

\begin{corollary} \label{4.5} With the notation of Theorem \ref{4.4} the ring $B$
is a semi-local ring. The number of maximal ideals of $B$ is equal to the number of connected components of $\mathbb G_{R/I}.$
\end{corollary}

\begin{proof} First $B$ is a semi-local ring as follows because it is a finitely
generated $R$-module. If $\mathbb G_{R/I}$ is connected in codimension one, then
$B$ is a local ring (cf. Theorem \ref{4.4}). Then the claim follows as in the proof
of Theorem \ref{4.4} by $H^c_I(R) \simeq \oplus_{i=1}^t H^c_{I_i}(R).$ Here $I_i, i = 1, \ldots, t,$ denotes  the intersection of all
minimal primes of $V(I)$ that are vertices of $\mathbb G_i,$ the connected components of
$\mathbb G_{R/I}.$
\end{proof}

\begin{remark} \label{4.6} The proof of the statements (b) and (c) of Theorem \ref{1.1}
follows by Lemma \ref{4.2} and Remark \ref{4.3}. The claim (d) of Theorem \ref{1.1} is shown in Theorem \ref{4.4}. Finally (c) is a consequence of the results by Lyubeznik (cf. \cite{gL2})
and Zhang (cf. \cite{wZ}).
\end{remark}

\section{On Lyubeznik numbers}
In this section let $(R,\mathfrak m)$ be a regular local ring
containing a field. Let $I \subset R$ be an ideal of $R.$ We will
add a few results concerning the Lyubeznik numbers
\[
\dim_k \Hom_R(k, H^j_{\mathfrak m}(H^i_I(R))) = \dim_k \Ext_R^j(k,
H^i_I(R))
\]
in general. That means, we are interested in them for all pairs
$(j,i)$ not necessary for $(j,i) = (d,c).$ As a first step towards
this direction we improve the estimate in Remark \ref{4.1} (b).
To this end we use the Hartshorne-Lichtenbaum Vanishing Theorem.
It yields that $H^n_I(R) = 0, n = \dim R,$ whenever $(R,
\mathfrak m)$ is a complete local domain and $I$ is an ideal with
$\dim R/I > 0$ (cf. \cite[Theorem 3.1]{rH3} or \cite[Theorem
2.20]{pS1}).

\begin{lemma} \label{5.1} Let $(R,\mathfrak m)$ denote a regular
local ring with $\dim R = n.$ Let $I$ be an ideal of $R$ and
$c = \height I < n.$ Then
\[
\dim H^i_I(R) \leq n - i - 1
\]
for all $c < i \leq n$ and $\dim H^c_I(R) = n - c.$
\end{lemma}

\begin{proof} Let $\mathfrak p \in V(I)$ denote a prime ideal such
that $\dim R_{\mathfrak p} = c.$ Then
\[
0 \not= H^c_{I R_{\mathfrak p}}(R_{\mathfrak p}) \simeq H^c_I(R)
\otimes_R R_{\mathfrak p}
\]
because $c = \dim R_{\mathfrak p}$ and $\Rad IR_{\mathfrak p} =
\mathfrak p R_{\mathfrak p}.$ Recall that $\mathfrak p \in V(I)$
is a minimal prime ideal. Therefore $ \dim H^c_I(R) \geq \dim
R/\mathfrak p = n-c.$ The equality is true since $\Supp H^c_I(R)
\subseteq V(I).$

Now let $c < i \leq n.$ First of all note that $H^n_I(R) = 0$ as
follows by the Hartshorne-Lichtenbaum Vanishing Theorem (cf.
\cite{rH3} or \cite[Theorem 2.20]{pS1}). Now suppose the contrary
to the claim. That is $\dim  H^i_I(R) \geq n - i$ for a certain $c
< i < n.$ Then there is a prime ideal $\mathfrak p \in \Supp
H^i_I(R)$ such that $\dim R/\mathfrak p \geq n - i.$ Therefore
$H^i_{I R_{\mathfrak p}}(R_{\mathfrak p}) \not= 0$ and $i < \dim
R_{\mathfrak p}$ as follows again by the Hartshorne-Lichtenbaum
Vanishing Theorem. Moreover $\dim R/\mathfrak p = n - \dim
R_{\mathfrak p} \geq n -i,$ and therefore $\dim R_{\mathfrak p}
\geq i,$ which is a contradiction.
\end{proof}

For a better understanding of the Lyubeznik numbers we need an
auxiliary construction of a certain complex $C^{\cdot}(I)$ for an
ideal $I$ of $R.$ Here we assume that $(R,\mathfrak m)$ is a Gorenstein
ring.
Let $R \qism E^{\cdot}$ denote a minimal injective resolution of
$R.$ Because of $\Gamma_I(E^{\cdot})^i = 0$ for all $i < c = \height I < n$ there
is a homomorphism of complexes
\[
0 \to H^c_I(R)[-c] \to \Gamma_I(E^{\cdot}),
\]
where $H^c_I(R)$ is considered as a complex concentrated in
homological degree zero.

\begin{definition} \label{5.2} The cokernel of the embedding  $H^c_I(R)[-c] \to \gam (E^{\cdot})$
is defined as $C^{\cdot}_R(I),$ the truncation complex with respect to $I.$ So there is a short
exact sequence of complexes of $R$-modules
\[
0 \to H^c_I(R)[-c] \to \gam (E^{\cdot}) \to C^{\cdot}_R(I) \to 0.
\]
We observe that $H^i(C^{\cdot}_R(I)) \simeq H^i_I(R)$ for all $i \not=
c$ while $H^c(C^{\cdot}_R(I)) = 0.$
\end{definition}

By applying the derived functor $\Rgam$ of the
section functor it induces (in the derived category) a
short exact sequence of complexes
\[
0 \to \Rgam(H^c_I(R))[-c] \to E[-n] \to \Rgam(C^{\cdot}(I)) \to 0.
\]
Recall that $\Rgam(\Gamma_I(E^{\cdot})) \simeq \Gamma_{\mathfrak
m}(\Gamma_I(E^{\cdot})) \simeq \Gamma_{\mathfrak m}(E^{\cdot})
\simeq E[-n],$ where $E = E_R(R/\mathfrak m)$ denotes the
injective hull of the residue field.

In order to compute the hypercohomology $H^i_{\mathfrak
m}(C^{\cdot}(I))$ there is the following $E_2$-term spectral
sequence (cf. \cite{cW} for the details)
\[
E_2^{p,q} = H^p_{\mathfrak m}(H^q(C^{\cdot}(I))) \Rightarrow
E^{p+q}_{\infty} = H^{p+q}_{\mathfrak m}(C^{\cdot}(I)).
\]
Now recall that $H^q(C^{\cdot}(I)) = H^q_I(R)$ for $c < q <n,$ and
$H^q(C^{\cdot}(I)) = 0$ for $q \leq c$ resp. $q \geq n.$ We notice
that $H^n_I(R) = 0$ as a consequence of the Hartshorne-Lichtenbaum
Vanishing Theorem.

\begin{proposition} \label{5.3} Let $(R,\mathfrak m)$ be a regular local
ring. With notation of Definition \ref{5.2} there
is a short exact sequence
\[
0 \to H^{n-1}_{\mathfrak m}(C^{\cdot}(I)) \to H^d_{\mathfrak
m}(H^c_I(R)) \to E \to 0
\]
and isomorphisms $H^{j-1}_{\mathfrak m}(C^{\cdot}(I)) \simeq
H^{j-c}_{\mathfrak m}(H^c_I(R))$ for all $j < n$ and all $j > n =
\dim R.$
\end{proposition}

\begin{proof} The proof follows by the long exact cohomology
sequence of the above short exact sequence of complexes in the derived
category. The only
claim we have to show is the vanishing of $H^n_{\mathfrak
m}(C^{\cdot}(I)).$ To this end we apply the previous spectral
sequence. By virtue of Lemma \ref{5.1} we know that $\dim
H^q_I(R) < n - q$ for all $q \not= c.$ Therefore $E^{n-q,q}_2 =
H^{n-q}_{\mathfrak m}(H^q_I(R)) = 0$ for all $q \not= c.$ But this
provides the vanishing of $H^n_{\mathfrak m}(C^{\cdot}(I)),$ as
required.
\end{proof}

As an application of our investigations we prove a
slight improvement of a duality result shown by Blickle (cf.
\cite[Theorem 1.1]{mB}).

\begin{corollary} \label{5.4} Let $I \subset R$ denote an ideal
of a regular local ring $(R, \mathfrak m)$ containing a field.
Suppose that $\height c < n$ and  $\Supp H^i_I(R) \subset V(\mathfrak m)$ for all $i
\not= c.$ Then the following is true:
\begin{itemize}
\item[(a)] There is a short exact sequence
\[
0 \to H^{n-1}_I(R) \to H^d_{\mathfrak
m}(H^c_I(R)) \to E \to 0.
\]
\item[(b)] For $j < n$ there are isomorphisms $H^{j-1}_I(R)
 \simeq H^{j-c}_{\mathfrak m}(H^c_I(R)).$
\end{itemize}
\end{corollary}

\begin{proof} Because of $\Supp H^i_I(R) \subset V(\mathfrak m)$
for all $i \not= c$ it follows that $H^p_{\mathfrak
m}(H^q(C^{\cdot}(I)) = 0$ for all $p \not= 0.$ So the previous
spectral sequence degenerates to isomorphisms
\[
H^q_{\mathfrak m}(C^{\cdot}(I)) \simeq H^0_{\mathfrak m}(H^q(C^{\cdot}(I))
\]
for all $q \in \mathbb Z.$ Then the claim is a consequence of the
statements in Proposition \ref{5.3}. Recall that
$H^0_{\mathfrak m}(H^q(C^{\cdot}(I) \simeq H^q(C^{\cdot}(I))$ for all
$q \in \mathbb Z$ because $\Supp H^i_I(R) \subset V(\mathfrak m)$ for all $i
\not= c$ by the assumption.
\end{proof}

Under the assumption of Corollary \ref{5.4} Blickle (cf.
\cite[Theorem 1.1]{mB}) proved the following equalities
\[
\dim_k \Hom_R(k,H^0_{\mathfrak m}(H^{j-1}_I(R))) =
\dim_k \Hom_R(k, H^{j-c}_{\mathfrak m}(H^c_I(R)))  - \delta_{j-c,d}
\]
for all $j \in \mathbb N.$ In fact, this is a consequence of the
present Corollary \ref{5.4}. The assumption $\Supp H^i_I(R) \subset V(\mathfrak m)$ for all $i \not= c$ is fulfilled whenever $\cd IR_{\mathfrak p} = \height I$ for all $\mathfrak p \in
V(I) \setminus \{\mathfrak m\}.$

Another result of this type is the following.

\begin{corollary} \label{5.5} With the notation of \ref{5.4}
let $I$ denote an ideal of $R.$ Suppose there is an integer $c < a < n$
such that $H^i_I(R) = 0$ for all $i \not= c, a.$

\begin{itemize}
\item[(a)] There is a short exact sequence
\[
0 \to H^{n-a-1}_{\mathfrak m}(H^a_I(R)) \to H^d_{\mathfrak
m}(H^c_I(R)) \to E \to 0.
\]
\item[(b)] For $j < n$ there are isomorphisms $H^{j-a-1}_{\mathfrak m}(H^a_I(R))
 \simeq H^{j-c}_{\mathfrak m}(H^c_I(R)).$
\end{itemize}
\end{corollary}

\begin{proof} By the assumption on the vanishing of $H^i_I(R)$ it
follows that $C^{\cdot}(I) \qism H^a_I(R)[-a]$ in the derived
category. Thus, the statement is an immediate consequence of
Proposition \ref{5.3}.
\end{proof}

It would be of some interest to get an understanding of the Lyubeznik numbers in
general.

\section{On a trace map}
Let $(R, \mathfrak m)$ denote a Gorenstein ring and $n = \dim R.$ Let
$I \subset R$ denote an ideal of $\height I = c$ and $\dim R/I = d.$

\begin{lemma} \label{6.1} With the previous notation there is a natural
homomorphism
 \[
 \phi :  \Ext^d_R(k, H^c_I(R)) \to k  .
 \]
\end{lemma}

\begin{proof} Apply the derived functor $\operatorname{R} \Hom_R(k,\cdot)$
to the short exact sequence as it is defined in the definition of the
truncation complex (cf. \ref{5.2}). Then there is the following short
exact sequence of complexes in the derived category
\[
 0 \to \operatorname{R} \Hom_R(k,H^c_I(R))[-c] \to \operatorname{R}
 \Hom_R(k,\Gamma_{\mathfrak m}(E^{\cdot})) \to \operatorname{R} \Hom_R(k,C^{\cdot}_R(I)) \to 0.
\]
Now we consider the complex in the middle. It is represented by
$\Hom_R(k, \Gamma_{\mathfrak m}(E^{\cdot}))$ since $\Gamma_{\mathfrak m}(E^{\cdot})$
is a complex of injective modules. Moreover there are the following isomorphisms
\[
 \Hom_R(k, \Gamma_{\mathfrak m}(E^{\cdot})) \simeq \Hom_R(k, E^{\cdot}) \simeq k[-n].
\]
By virtue of the long exact cohomology sequence it yields the natural homomorphism of the statement.
\end{proof}

In \cite[Conjecture 2.7]{HeS} Hellus and the author conjectured that the homomorphism in
Lemma \ref{6.1} is in general non-zero. In the following we shall confirm this
question in the case of $(R,\mathfrak m)$ a regular local ring containing a field. To this
end we need a few auxiliary constructions.

By virtue of Proposition \ref{5.3} and Lemma \ref{6.1} there is the following commutative
diagram
\[
\begin{array}{ccc}
  \Ext^d_R(k, H^c_I(R)) & \stackrel{\phi}{\longrightarrow} & k \\
  \downarrow  \lambda &  & \downarrow \\
  H^d_{\mathfrak m}(H^c_I(R)) & \stackrel{\psi}{\longrightarrow} & E .
\end{array}
\]
Here the vertical homomorphism $k \to E$ is -- by construction -- the natural inclusion. Therefore, $\lambda$ is not zero, provided $\phi$ is not zero.

\begin{theorem} \label{6.2} Let $(R, \mathfrak m)$ denote a regular local
ring containing a field with $\dim R = n.$ Let
$I \subset R$ denote an ideal of $\height I = c$ and $\dim R/I = d.$
Then the homomorphism
\[
\phi :  \Ext^d_R(k, H^c_I(R)) \to k
\]
is non-zero.
\end{theorem}

\begin{proof} We may assume that $R$ is a complete local ring. By applying
$\Hom_R(k, \cdot)$ to the above diagram it provides the following commutative diagram
\[
\begin{array}{ccc}
  \Ext^d_R(k, H^c_I(R)) & \stackrel{\phi}{\longrightarrow} & k \\
  \downarrow \bar{\lambda} &  & \parallel \\
  \Hom_R(k, H^d_{\mathfrak m}(H^c_I(R))) & \stackrel{\bar{\psi}}{\longrightarrow} & k.
\end{array}
\]
Because of Remark \ref{4.1} the vertical homomorphism $\bar{\lambda}$ is an
isomorphism. Therefore it will be enough to show that $\bar{\psi}$ is not zero.
By Matlis duality it follows that
\[
B = \Ext^c_R(H^c_I(R), R) \simeq \Hom_R(H^d_{\mathfrak m}(H^c_I(R)), E).
\]
So, $\bar{\psi}$ is the Matlis dual of the natural homomorphism
$k \otimes R \to k \otimes B$ which is non-zero as shown in Theorem \ref{3.1} (c).
This proves that $\bar{\psi}$ is non-zero. By the previous observation it follows
that $\phi$ is non-zero too.
\end{proof}

\begin{remark} \label{6.3} Let $(R, \mathfrak m)$ be local ring that is the factor ring
of a Gorenstein ring. Let $M$ be a finitely generated $R$-module with $d = \dim_R M.$ In connection to his canonical element conjecture (cf.  \cite[Section 4]{mH}) Hochster has studied the natural homomorphism $\Ext_R^d(k,K(M)) \to H^d_{\mathfrak m}(K(M)),$ where $K(M)$ denotes the canonical module of $M.$ In particular, he considered the problem whether this map is non-zero.

In our situation here the natural homomorphism $\lambda: \Ext^d_R(k, H^c_I(R)) \to H^d_{\mathfrak m}(H^c_I(R))$ is the direct limit of the natural homomorphisms
$\lambda_{\alpha} : \Ext_R^d(k, K(R/I^{\alpha})) \to H^d_{\mathfrak m}(K(R/I^{\alpha})), \alpha \in \mathbb N.$ Recall that $H^c_I(R) \simeq \varinjlim K(R/I^{\alpha})$ and that local cohomology commutes with direct limits. So in a certain sense, $\lambda$ is the stable value of all of the $\lambda_{\alpha}.$ It would be of some interest to relate the non-vanishing of $\lambda$ to other problems in commutative algebra.
\end{remark}

\section{Examples}

Let $I \subset R$ denote an ideal with $c = \height I.$
The following example shows that the property $\Hom_R(H^c_I(R),H^c_I(R)) \simeq R$ is
not preserved by passing to the localization with respect to a prime ideal.

\begin{example} \label{7.1} Let $k$ denote an algebraically closed field. Let $R =
k[|a,b,c,d,e|]$ denote the formal power series ring in five
variables. Let $I \subset R$ denote the prime ideal with the
parametrization
\[
a = su^2, b = stu, c = tu(t-u), d = t^2(t-u), e = u^3.
\]
It is easy to see that $I = (ad - bc, a^2c + abe -b^2e, c^3 + cde
- d^2e, ade - bde + ac^2).$ Moreover $\dim R/I = 3, n = 5$ and
$H^i_I(R) = 0$ for all $i \not= 2,3$ as it is a consequence of the
Second Vanishing Theorem. Clearly $V(I)$ is connected
in codimension one because $I$ is a prime ideal. So it follows (cf.
Remark \ref{4.3}) $\dim_k \Ext_R^3(k, H^2_I(R)) = 1.$ Therefore
the endomorphism ring $\Hom_R(H^2_I(R), H^2_I(R))$
is isomorphic to $R$ (cf. Lemma \ref{4.2}).

Let $\mathfrak p = (a,b,c,d)R.$ Then $\dim R_{\mathfrak p}/I
R_{\mathfrak p} = 2.$ The ideal $I R_{\mathfrak p}$ corresponds to
the parametrization $(x, xy, y(y-1), y^2(y-1)).$ It follows that
$V(I R_{\mathfrak p})\setminus \{{\mathfrak p}R_{\mathfrak p}\}$
is not formally connected (cf. \cite[3.4.2]{rH}). Therefore $V(I
\widehat{R_{\mathfrak p}})$ has two connected components. Whence
$\Hom_{R_{\mathfrak p}}(H^2_{I R_{\mathfrak p}}(R_{\mathfrak p}),
H^2_{I R_{\mathfrak p}}(R_{\mathfrak p})) \simeq {\widehat {R_{\mathfrak p}}}^2,$
because $\dim_{k(\mathfrak p)} \Ext_{R_{\mathfrak p}}^2(k(\mathfrak p), H^2_{I R_{\mathfrak p}}(R_{\mathfrak p}) = 2.$
\end{example}

The following example (invented by Hochster) shows that the Bass numbers
$\dim_k \Ext_R^i(k, H^c_I(R))$ depend upon the characteristic of the
ground field $k.$

\begin{example} \label{7.2} Let $R = k[|x_1,\ldots,x_6|]$ denote the
formal power series rings in six variables over the basic field $k.$
Let $I$ denote the ideal generated by the two by two minors of the
matrix
\[
M = \left(
      \begin{array}{ccc}
        x_1 & x_2 & x_3 \\
        x_4 & x_5 & x_6 \\
      \end{array}
    \right).
\]
Then $R/I$ is a four dimensional Cohen-Macaulay ring and $c = \height I =2.$ 
It follows that
$H^i_I(R) = 0$ for all $i \not= 2$ provided $k$ is a field of characteristic
$p > 0.$ Furthermore $H^i_I(R) = 0$ for all $i \not= 2, 3$ and $H^3_I(R) \not= 0$
in the case of $k$ a field of characteristic zero. This is shown via the
Reynolds operator. Clearly $R/I$ has an isolated singularity, so that $\Supp H^3_I(R) \subset
\{\mathfrak m\}.$ By virtue of Corollary \ref{5.4} and Remark \ref{4.1}
it follows that $\Ext_R^i(k, H^2_I(R)) = 0$ for all $i < 4$ if $k$ is of positive
characteristic, while $\Ext_R^2(k, H^2_I(R)) \not= 0$ if $k$ is of 
characteristic zero. 

Clearly  $H^0_{\mathfrak m}(H^3_I(R)) \simeq H^3_I(R),$ and therefore $H^3_I(R)$ 
is an injective $R$-module (cf. Remark \ref{4.1}). Now we apply the derived 
functor $\operatorname{R}\Hom_R(k, \cdot)$ to the short exact sequence of 
the truncation complex (cf. Definition \ref{5.2}). Because of $H^3(C^{\cdot}_R(I)) 
\simeq H^3_I(R)$ and $H^i(C^{\cdot}_R(I)) = 0$ for all $i \not= 3$ it induces an 
isomorphism $\Ext^2_R(k,H^2_I(R)) \simeq \Hom_R(k, H^3_I(R)).$ 

Finally Uli Walther (cf. \cite[Example 6.1]{uW}) has computed that $H^3_I(R) \simeq 
E_R(k).$ Therefore $\dim \Ext_R^2(k, H^2_I(R)) = 1$ in the case of characteristic zero.
\end{example}

The Example \ref{7.3} shows that the number of maximal
ideals of $B = \Hom_R(H^c_I(R), H^c_I(R))$ does not coincide with the multiplicity
of $B.$  

\begin{example} \label{7.3} Let $\mathbb Q$ denote the field of rational numbers.
Consider $\mathbb Q(i)$ denote the field extension of $\mathbb Q$ by the imaginary
unit. Let $R = \mathbb Q[|w,x,y,z|]$ and $S = \mathbb Q(i)[|w,x,y,z|]$ denote
the formal power series ring in four variables over $\mathbb Q$ and $\mathbb Q(i)$
respectively. Let $J = (w-ix,y-iz) \cap (w+ix, y+iz) \subset S$ and $I = J \cap R.$
Then $I = (w^2+x^2,y^2+z^2,wy+xz,wz-xy)$ is a two-dimensional prime ideal. Therefore (cf.
Theorem \ref{4.4}) $B_R = \Hom_R(H^2_I(R), H^2_I(R))$ is a local ring. Moreover
\[
B_S = \Hom_S(H^2_J(S), H^2_J(S)) \simeq S/(w-ix,y-iz) \oplus S/(w+ix,y+iz)
\]
as it follows by the Mayer-Vietoris sequence for local cohomology. It is easily
seen that $B_R \simeq \mathbb Q[a]/(a^2+1)[|w,x,y,z|].$ In fact, $\dim \Ext_R^2(R/\mathfrak m, H^c_I(R) = 2.$
\end{example}

\begin{acknowledgement} The author is grateful to Wenliang Zhang for drawing his attention to Uli Walther's 
paper \cite{uW} and suggesting the computation of the Lyubeznik number in Example \ref{7.2} in the 
situation of characteristic zero. The final example was suggested to the author by Lyubeznik in order to correct
an error in a previous version of the paper.
\end{acknowledgement}

\end{document}